\documentclass[12pt]{article}
\usepackage{amsmath,amsthm,amsfonts,amssymb}
\usepackage{graphicx}
\usepackage[utf8]{inputenc}
\usepackage{times}
\usepackage[T1]{fontenc}
\def\C{{\mathbb C}}
\def\D{{\mathbb D}}

\def\R{{\mathbb R}}

\def\T{{\mathbb T}}
\def\Z{{\mathbb Z}}
\def\CP{{\mathbb{CP}}}
\def\cO{{\mathcal{O}}}
\def\U{\mathrm{U}}
\def\dvf#1{{\underrightarrow{#1}}} 
\def\ddC{{d\,d^{\scriptscriptstyle\C}}}
\def\dC{{d^{\scriptscriptstyle\C}}}
\def\ddJ{{d\,d^{\scriptscriptstyle J}}}
\def\dJ{{d^{\scriptscriptstyle J}}}
\def\etc{,\dots,}
\def\wt{\widetilde}
\def\wh{\widehat}
\def\ol{\overline}

\def\map{\colon}
\def\del{\partial}
\def\eps{\varepsilon}
\def\pii{\pi i}
\def\rst#1{{\upharpoonright}_{\!#1}}
\def\scup{\mathop{\smallsmile}}

\def\with{\mathrel{:}}
\def\id{\mathrm{id}}
\def\hook{\mathbin{\lrcorner}}
\def\Step#1){\noindent\hbox to \parindent{#1)\hfill}}
\theoremstyle{plain}
\newtheorem{theorem}{Theorem}
\newtheorem{corollary}[theorem]{Corollary}
\newtheorem{proposition}[theorem]{Proposition}
\newtheorem{lemma}[theorem]{Lemma}

\theoremstyle{definition}
\newtheorem{definition}[theorem]{Definition}

\theoremstyle{remark}
\newtheorem*{acknowledgments}{Acknowledgments}
\newtheorem*{remark*}{Remark}
\newtheorem{remark}[theorem]{Remark}

\def\cC{{\mathcal{C}}}
\DeclareMathOperator{\Sk}{Sk}
\DeclareMathOperator{\Cl}{Clos}

\begin{document}

\title{Remarks on Donaldson's symplectic submanifolds}
\author{Emmanuel \textsc{Giroux}\thanks%
{\ \emph{Centre National de la Recherche Scientifique} (UMI 3457 CRM) and \emph
{Université de Montréal}. Partially supported by the ANR grant \emph{MICROLOCAL}
(ANR-15CE40-0006).}}
\date{Montreal --- December 2017}

\maketitle

In \cite{Do1}, S.~Donaldson proved the following:

\begin{theorem}[Donaldson] \label{t}
Let $V$ be a closed manifold and $\omega$ a symplectic form on $V$ with integral
periods. Then, for every sufficiently large positive integer $k$, there exists a
symplectic submanifold $W$ of codimension $2$ in $(V,\omega)$ whose homology
class is Poincaré dual to $k[\omega]$ and whose inclusion into $V$ is an
$(n-1)$-connected map, where $n := \frac12 \dim_\R V$.
\end{theorem}

This result highlights analogies between symplectic geometry and Kähler geometry
which were quite unexpected at the time, and actually the ideas and the methods
introduced by Donaldson in \cite{Do1, Do2} provide a new insight into both
fields. When $V$ is a complex projective manifold and $\omega$ a Kähler form
with integral periods, the above theorem is a classical result that follows from
the works of Bertini, Kodaira and Lefschetz. In this case, $W \subset V$ is a
complex hypersurface obtained as a transversal hyperplane section $V \cap H$ of
$V$, where $V$ is holomorphically embedded into a projective space $\CP^m$ and
$H \subset \CP^m$ is a hyperplane. As a consequence, $V - W \subset \CP^m - H
\simeq \C^m$ is a smooth affine variety and, in particular, a Stein manifold of
finite type. Moreover, $\omega \rst{ V - W } = \ddC\phi$ for some exhausting
function $\phi \map V-W \to \R$ having no critical points near $W$. Explicitly,
$\phi := -\frac1{2k\pi} \log|s|$ where $s$ is the restriction to $V \subset \CP
^m$ of the complex linear function (a holomorphic section of $\cO(1)$) defining
$H$. (Note that the operator $\dC$ here is given by $\dC\phi(v) := - d\phi(iv)$
for any tangent vector $v$.) 
 
Our main purpose in this paper is to show that any closed integral symplectic
manifold has a very similar structure:

\begin{theorem}[Stein Complements] \label{t:sc}
Let $V$ be a closed manifold and $\omega$ a symplectic form on $V$ with integral
periods. Then, for every sufficiently large positive integer $k$, there exist:
\begin{itemize}
\item 
a symplectic submanifold $W$ of codimension $2$ in $(V,\omega)$ whose homology
class is Poincaré dual to $k[\omega]$, and 
\item
a complex structure $J$ on $V - W$ such that $\omega \rst{ V - W } = \ddJ\phi$
for some exhausting function $\phi \map V - W \to \R$ having no critical points
near $W$; in particular, $(V - W, J)$ is a Stein manifold of finite type.
\end{itemize}
\end{theorem} 

Of course, the difference with the Kähler case is that, in general, the complex
structure $J$ (which depends on $k$) does not extend over the submanifold $W$.  
To make the above statement less mysterious, we need to recall a few pieces of
terminology.

A \emph{Liouville domain} is a domain%
\footnote{In this text, the word \emph{domain} means ``compact manifold with
boundary.''}
$F$ endowed with a \emph{Liouville form}, namely, a $1$-form $\lambda$ with the
following properties:
\begin{itemize}
\item
$d\lambda$ is a symplectic form on $F$, and
\item
$\lambda$ induces a contact form on $K := \del W$ orienting $K$ as the boundary
of $(F,d\lambda)$; equivalently, the \emph{Liouville vector field} $\dvf\lambda$
given by $\dvf\lambda \hook d\lambda = \lambda$ points transversely outwards
along $K$. 
\end{itemize}
A Liouville domain $(F,\lambda)$ is a \emph{Weinstein domain} if the Liouville
field $\dvf\lambda$ is gradientlike for some Morse function $\phi \map F \to
\R$, meaning that
$$ \dvf\lambda \cdot \phi \ge c\, |\dvf\lambda|^2, $$
where the norm is computed with respect to any auxiliary metric and $c$ is a
positive number depending on that metric. (Obviously, the function $\phi$ can be
further adjusted to be constant on $\del F$.)

Not every Liouville domain is a Weinstein domain. In fact, no restriction is
known for the topology of a Liouville domain while the topology of a Weinstein
domain is strongly constrained. More explicitly, the topology of a Liouville
domain $(F,\lambda)$ is largely concentrated in its skeleton (also called core,
or spine), namely the union $\Sk(F,\lambda)$ of all the orbits of $\dvf\lambda$
which do not exit through $\del F$. Indeed, the whole domain retracts onto an
arbitrary small neighborhood of $\Sk(F,\lambda)$. Due to the dilation properties
of $\dvf\lambda$ (its flow expands $\lambda$ exponentially), the closed subset
$\Sk(F,\lambda) \subset F$ has measure zero (for the volume form $(d\lambda)^n$,
where $n := \frac12 \dim F$), but for instance there are Liouville domains
$(F,\lambda)$ for which $\Sk(F,\lambda)$ is a stratified subset of codimension
$1$ \cite{Mc, Ge, MNW}. In contrast, if $(F,\lambda)$ is a Weinstein domain, 
$\Sk(F,\lambda)$ consists of the stable submanifolds of the critical points of
the Lyapunov function $\phi$. Then the same dilation properties as above force
these submanifolds to be isotropic for $d\lambda$, and so the critical indices
of $\phi$ cannot exceed $n$. In particular, the inclusion $\del F \to F$ is an
$(n-1)$-connected map. Actually, the main examples of Weinstein domains are
Stein domains, \emph{i.e.}, sublevel sets of exhausting $\C$-convex%
\footnote{We use the term $\C$-convex ---~or $J$-convex, if we want to refer to
a specific complex structure $J$~--- to mean ``strictly plurisubharmonic.''}
functions, and the work of Cieliebak-Eliashberg \cite{CE} shows that Weinstein
and Stein domains are essentially the same objects. As for the relationships
between Weinstein and Liouville domains, they remain quite mysterious.

Returning to our closed integral symplectic manifold $(V,\omega)$, we will call
\emph{hyperplane section of degree $k$ in $(V,\omega)$} any submanifold $W$ of
codimension $2$ in $V$ whose homology class is Poincaré dual to $k[\omega]$. A
preliminary remark is that the complement of a symplectic hyperplane section $W$
of arbitrary degree in $(V,\omega)$ is isomorphic to the interior of a Liouville
domain (cf.~Proposition \ref{p:lc}). There is no general evidence that
the Liouville domains obtained in this way have peculiar topological properties,
but this may happen under additional assumptions on $(V,\omega)$. Revisiting a 
construction due to Auroux \cite{Au1}, we will illustrate this by discussing the
case of symplectic hyperplane sections in tori (see Propositions \ref{p:T4} and
\ref{p:T2n}). As for the symplectic hyperplane sections provided by Donaldson's
construction, we have (see \cite[Proposition 8]{Gi}):

\begin{theorem}[Weinstein Complements] \label{t:wc}
Let $V$ be a closed manifold and $\omega$ a symplectic form on $V$ with integral
periods. Then, for every sufficiently large positive integer $k$, there exist a
Weinstein domain $(F,\lambda)$ and a map $q \map F \to V$ with the following
properties:
\begin{itemize}
\item
$q(\del F)$ is a symplectic hyperplane section $W$ of degree $k$ in $(V,\omega)$
and $\del F$ is the normal circle bundle of $W$ projecting to $W$ by $q$;
\item
$q \rst{ F - \del F } \map F - \del F \to V - W$ is a diffeomorphism, with
$q^*\omega = d\lambda$.
\end{itemize}
\end{theorem}

Theorem \ref{t:sc} is then a corollary of Theorem \ref{t:wc} and
the results of \cite{CE}.

\begin{remark}[About Tiles]
The proofs of Theorems \ref{t:sc} and \ref{t:wc} are variants of
Donaldson's proof of Theorem \ref{t}. In particular, with the terminology used
by Auroux in \cite{Au2}, the symplectic hyperplane sections they produce are the
zero sets of ``asymptotically holomorphic and uniformly transverse sections'' of
certain prequantization line bundles. It then follows from Auroux's uniqueness
theorem \cite[Theorem 2]{Au2} that, for every sufficiently large integer $k$,
these various symplectic hyperplane sections lie in the same Hamiltonian isotopy
class. Thus, Theorems \ref{t:sc} and \ref{t:wc} can essentially be
rephrased by saying that the symplectic hyperplane sections given by Donaldson's
construction have Stein, resp.~Weinstein, complements.     
\end{remark}

In \cite{Bi}, Biran adopted a very fruitful new viewpoint on the decomposition
of a complex projective manifold $V$ described at the beginning of this paper.
Instead of regarding $V$ as decomposed into a complex hyperplane section $W$ and
the affine variety $V - W$, he considered $V$ as consisting of the skeleton of
$V - W$ (this Stein manifold can be compactified to a Weinstein domain) and its
complement. His key observation is that the latter is a simple symplectic object
that he calls a ``standard symplectic disk bundle'' over $W$ (see the discussion
preceding Corollary \ref{c:sdb} for a precise definition). As a byproduct of
Theorem \ref{t:wc}, we can extend Theorem 1.A of \cite{Bi} as follows:

\begin{corollary}[Generalization of Biran's Decomposition] \label{c:bi-1.A}
Let $V$ be a closed manifold and $\omega$ a symplectic form on $V$ with integral
periods. Then, for every sufficiently large positive integer $k$, there exists 
an isotropic skeleton $\Delta \subset V$ whose complement $V - \Delta$ has the
structure of a standard symplectic disk bundle of area $1/k$ over a symplectic
manifold $W$. 
\end{corollary}

Actually, one can take for $\Delta$ the skeleton of any Weinstein domain as in
Theorem \ref{t:wc}. We refer the reader to \cite{Bi} for applications
of Corollary \ref{c:bi-1.A} to intersection problems. 

\begin{acknowledgments}
I wish to thank Jean-Paul Mohsen for the multiple exchanges we had over the
years about Donaldson's work on complex and symplectic geometry; his approach to
the theory (see for instance \cite{Mo2}) was strongly influential. I also thank
Hélène Eynard-Bontemps for fixing many misprints in the manuscript.
\end{acknowledgments}

\subsection{Symplectic hyperplane sections and Liouville domains} \label{s:lc}

We begin with a simple observation:

\begin{proposition}[Liouville Complements] \label{p:lc}
Let $V$ be a closed manifold, $\omega$ a symplectic form on $V$ with integral
periods and $W \subset V$ a symplectic hyperplane section of degree $k$. Then
there exists a Liouville domain $(F,\lambda)$ and a map $q \map F \to V$ with
the following properties:
\begin{itemize}
\item
$q(\del F) = W$ is the symplectic hyperplane section, $\del F$ is the normal
circle bundle of $W$ projecting to $W$ by $q$, and $-2k\pii\lambda$ defines
a unitary connection on $\del F$ with curvature form $-2k\pii\omega \rst W$;
\item
$q \rst{ F - \del F } \map F - \del F \to V - W$ is a diffeomorphism, and
$q^*\omega = d\lambda$.
\end{itemize}
\end{proposition}

A Liouville domain as above will be called a \emph{Liouville compactification}
of $V - W$.

\begin{remark}[Liouville Domains and Symplectic Hyperplane Sections] \label{r:lcomp}
Conversely, take a Liouville domain $(F,\lambda)$ whose boundary $\del F$ has
the structure of a principal circle bundle over a manifold $W$, and assume that
$-2k\pii\lambda$, for some positive integer $k$, induces a (unitary) connection 
form on $\del F$. Then the quotient $V$ of $F$ by the equivalence relation which
collapses every fiber of $\del F \to W$ to a point is an integral symplectic
manifold in which $W$ sits as a symplectic hyperplane section of degree $k$.
\end{remark}

\begin{proof}
Let $L \to V$ be a Hermitian line bundle whose Chern class is a lift of
$k[\omega]$, and denote by $P \subset L$ the unit circle bundle with projection
$p \map P \to V$. By standard obstruction theory, $L$ has a section $s$ whose
zero set equals $W$ and is cut out transversely. Then $u = s/|s|$ is a section
of $P$ over $V - W$, and the set
$$ F = u (V-W) \cup p^{-1}(W) = \Cl \bigl(u(V-W)\bigr) \subset P $$
is a smooth compact submanifold of $P$ with boundary $K := p^{-1}(W)$, which can
be viewed as the result of a ``real oriented blowup'' of $V$ along $W$. 

Fix a unitary connection $\nabla$ on $L$ with curvature form $-2k\pii\omega$. On
the principal $\U_1$-bundle $P$, the connection $\nabla$ is given by a $1$-form
$-2k\pii\alpha$ where $\alpha$ is a real contact form such that $d\alpha = p^*
\omega$. Thus, the $1$-form $\lambda$ induced by $\alpha$ on $F$ restricts to a
contact form on $K$, and satisfies
$$ u^*d\lambda = u^*d\alpha = u^*(p^*\omega) = (p \circ u)^* \omega = \omega. $$
Therefore, $(F,\lambda)$ is essentially the required Liouville domain, except
that $d\lambda$ degenerates along $K = \del F$ (the kernel of $d\alpha$ is 
spanned by the vector field generating the $\U_1$-action, and hence is tangent
to $K$). Lemma \ref{l:ld} below explains how to solve this problem by attaching
the boundary differently.
\end{proof}

Now recall that the \emph{symplectization} of a contact manifold $(K,\xi)$ is
the symplectic submanifold $SK$ of $T^*K$ consisting of the non-zero covectors
$\beta_x \in T^*_xK$, $x \in K$, whose cooriented kernel is $\xi_x$ (all contact
structures are cooriented in this paper). This is an $\R_{>0}$-principal bundle
over $K$ whose sections are the global Pfaff equations of $\xi$. Thus, any such
$1$-form $\alpha$ determines a splitting 
$$ SK = \bigl\{s \alpha_x \in T^*K \with (s,x) \in \R_{>0} \times K\bigr\} 
   \simeq \R_{>0} \times K . $$
We denote by $K_\alpha \subset SK$ the graph of $\alpha$, and by $SK_{<\alpha}$
(resp.\ $SK_{\le\alpha}$) the subset of $SK$ given by the condition $s<1$ (resp.\
$s \le 1$).

\begin{lemma}[Boundary Degenerations of Liouville Domains] \label{l:ld}
Let $F$ be a domain and $\lambda$ a $1$-form on $F$ which is a positive contact
form on $K := \del F$ and whose differential $d\lambda$ is a symplectic form on
$F - K$ but may degenerate along $K$. Then the singular foliation spanned by
$\dvf\lambda$ in $F - K$ extends to a foliation of $F$ transverse to $K$ and, 
denoting by $U$ the open collar consisting of all orbits which exit through $K$,
there exists a unique smooth homeomorphism
$$ h = h_\lambda \map U \to SK_{\le\alpha} $$
such that:
\begin{itemize}
\item
$h$ is the identity on $K \cong K_\alpha$ and induces a diffeomorphism between
$U - K$ and $SK_{<\alpha}$;
\item
$\lambda \rst U = h^*\lambda^\xi$ where $\lambda^\xi$ is the canonical $1$-form
on $SK$.
\end{itemize}
Furthermore, the singularities of $h$ are exactly the points of $K$ where 
$d\lambda$ degenerates and, in particular, the points where the $2n$-form
$(d\lambda)^n$ vanishes transversely (with $n := \frac12 \dim F$) correspond to
folds. 
\end{lemma}

As a result, one can change $(F,\lambda)$ to a genuine Liouville domain just by
gluing $F - K$ with $SK_{\le\alpha}$ along $U - K \cong SK_{<\alpha}$.

\begin{proof}
Let $\mu$ be an arbitrary positive volume form on $F$ and consider the function
$v := (d\lambda)^n / \mu$. We shall show that the vector field $\nu$ given by
$\nu \hook \mu = n \lambda \wedge (d\lambda)^{n-1}$ has the following
properties:
\begin{itemize}
\item
$\nu$ is non-singular along $K$ and points transversely outwards;
\item
$\nu = v \dvf\lambda$ at every point where $d\lambda$ is non-degenerate;
\item
the flow $f_t$ of $\nu$ is defined for all $t \le 0$ and the diffeomorphism
$$ f \map \R_- \times K \to U, \quad (t,x) \mapsto f_t(x), $$
satisfies $f^*\lambda = e^w \alpha$ where $w(t,x) = \int_0^t v(f_s(x))\, ds$.  
\end{itemize}
The first two properties show that $\nu$ generates a foliation transverse to $K$
which extends the foliation spanned by $\dvf\lambda$. The third property implies
that the map $h \map U \to SK_{\le\alpha}$ defined by
$$ h \circ f (t,x) = e^{ w(t,x) } \alpha_x $$
is a smooth homeomorphism with the desired behavior. Moreover, $h$ is unique
since the identity is the only homeomorphism of $SK_{\le\alpha}$ which fixes
$K_\alpha$ pointwise and induces a diffeomorphism of $SK_{<\alpha}$ preserving
$\lambda^\xi$.

The contact property of $\lambda$ means that $\lambda \wedge (d\lambda)^{n-1}$ 
induces a positive volume form on $K$, so $\nu$ is non-singular along $K$ and
points transversely outwards. Next, at any point where $d\lambda$ is symplectic,
$$ \dvf\lambda \hook d\lambda^n = v \dvf\lambda \hook \mu 
 = n \lambda \wedge (d\lambda)^{n-1} = \nu \hook \mu, $$
so $\nu = v \dvf\lambda$. In particular, $\nu \hook d\lambda = v \lambda$ and
this equality holds everywhere on $F$ by continuity.

To compute the form $f^*\lambda$, note that it vanishes on $\del_t$, $t \in
\R_-$, because $Df(\del_t) = \nu$ and $\nu \hook \lambda = 0$. Thus $f^*\lambda$
at a point $(t,x)$ is just (the pullback of) $f_t^*\lambda$ at point $x$.
Furthermore, $f_t^*\lambda$ satisfies the linear differential equation 
$$ \tfrac d{dt} f_t^*\lambda = f_t^* (\nu \cdot \lambda)
 = f_t^* (\nu \hook d\lambda) = f_t^* (v \lambda) 
 = (v \circ f_t) \, f_t^*\lambda . $$
Since $f_0^*\lambda = \alpha$, we obtain 
$$ f_t^* \lambda = \exp \left( \int_0^t (v \circ f_s)\, ds\right) \, \alpha, $$
as claimed.
\end{proof}

We now briefly describe the notion of standard symplectic disk bundle, referring
to \cite[Subsection 2.1]{Bi} for a more detailed discussion. The most relevant
approach here is as follows. Consider a closed integral symplectic manifold
$(W, \omega_W)$ and denote by $p \map K \to W$ a principal $\U_1$-bundle  whose
Chern/Euler class is an integral lift of $[\omega]$. Fix any connection $1$-form
$-2\pii\alpha$ on $K$ such that $d\alpha = p^*\omega_W$. Then $\alpha$ is a
contact form on $K$ and the quotient of the manifold $SK_{\le\alpha}$ that we
obtain by collapsing each circle fiber in $K = K_\alpha$ to a point has the
structure of an open disk bundle $U$ over $W$ and inherits a symplectic form
$\omega_U$ from $SK$ whose restriction to the zero section $W$ is $\omega_W$.
Moreover, each fiber of $U \to W$ is a symplectic disk of area $1$ (by Stokes'
theorem). The symplectic manifold $(U, \omega_U)$ is what Biran calls a \emph
{standard symplectic disk bundle of area $1$} over $W$ (see \cite[Remarks 2.1]
{Bi}). If the form $\frac1k \omega_W$ also has integral periods for some integer
$k \ge 1$ then $(U, \frac1k \omega_U)$ is named a \emph{standard symplectic disk
bundle of area $1/k$}.  

Given a Liouville domain $(F,\lambda)$ with boundary $K := \del F$, the manifold
$F - \Sk(F,\lambda)$, equipped with the $1$-form $\lambda$, is isomorphic to
$SK_{\le \lambda \rst K}$ with its canonical $1$-form. Thus, as a consequence of
Proposition \ref{p:lc}, we have:

\begin{corollary}[Standard Disk Bundles in Symplectic Manifolds] \label{c:sdb}
Let $V$ be a closed manifold, $\omega$ a symplectic form on $V$ with integral 
periods, $W$ a symplectic hyperplane section of degree $k$ and $(F,\lambda)$ a
Liouville compactification of $V - W$. Then the complement of $Sk(F,\lambda)$ in
$(V,\omega)$ has full measure and is a standard symplectic disk bundle of area
$1/k$.
\end{corollary}

Corollary \ref{c:bi-1.A} follows readily from Theorem \ref {t:wc} and Corollary
\ref{c:sdb}. 

\medskip

In the remainder of this section, we make a couple of remarks on the topology of
symplectic hyperplane sections in tori. We begin with an observation of Auroux
\cite{Au1, Au4} which shows that the Liouville domains given by Proposition \ref
{p:lc} need not be Weinstein domains:

\begin{proposition}[Auroux] \label{p:T4}
In the standard symplectic torus of dimension $4$, there exist disconnected 
symplectic hyperplane sections of arbitrarily large even degrees.
\end{proposition}

In particular, the complements of these symplectic hyperplane sections have
Liouville compactifications which are not Weinstein domains.

Interestingly enough, Auroux's argument can be ``reversed'' in higher dimensions
to prove the following:

\begin{proposition}[Connectedness in Higher Dimensional Tori] \label{p:T2n}
In the standard symplectic torus of dimension $2n \ge 6$, every symplectic
hyperplane section is connected.
\end{proposition} 

\begin{proof}[Proofs of Propositions \ref{p:T2n} and \ref{p:T4}]
The main underlying remark is that, if a closed integral symplectic manifold
$(V,\omega)$ of dimension $2n$ contains a disconnected symplectic hyperplane
section $W = W_1 \sqcup W_2$, then the cohomology class $w$ Poincaré dual to
$[W] = [W_1]+[W_2]$ splits as the sum $w_1+w_2$ of two non-zero integral classes
which satisfy $w_1 \scup w_2 = 0$ and $w_i \scup w^{n-1} > 0$, $i \in \{1,2\}$.
It follows that $w^n = w_1^n + w_2^n$, so either $w_1^n$ or $w_2^n$ is non-zero
(and positive). We assume below that $w_1^n > 0$.

If $V = \T^{2n} = \R^{2n} / \Z^{2n}$, its cohomology algebra can be identified
with the exterior algebra of $\R^{2n}$. In this identification, $w_1$ and $w_2$
become exterior $2$-forms $\omega_1$ and $\omega_2$, and the hypothesis that
$w_1^n > 0$ means that $\omega_1$ is a linear symplectic form. But then, by a
classical result of Lefschetz, multiplication by $\omega_1$ defines a map  
$\bigwedge^2 \R^{2n} \to \bigwedge^4 \R^{2n}$ which is injective for $n \ge 3$.
Since $\omega_1 \wedge \omega_2 = 0$, we get to the conclusion that $\omega_2 =
0$, which contradicts our assumption that $w_1$ and $w_2$ are non-zero. This
proves Proposition \ref{p:T2n}. 

To prove Proposition \ref{p:T4} (following Auroux \cite{Au1, Au4}), we first 
notice that the symplectic form $\omega := dx_1 \wedge dx_2 + dx_3 \wedge dx_4$
on $\T^4 = \R^4 / \Z^4$ can be written as $\omega = \frac12 (\omega_1+\omega_2)$
where $\omega_1, \omega_2$ are positive linear symplectic forms with integral
periods whose product $\omega_1 \wedge \omega_2$ is zero. For instance, one can
take 
\begin{align*}
   \omega_1 &:= dx_1 \wedge (dx_2-dx_3) + (dx_3+dx_2) \wedge dx_4, \\
   \omega_2 &:= dx_1 \wedge (dx_2+dx_3) + (dx_3-dx_2) \wedge dx_4.
\end{align*}
Next, we observe that the homology classes Poincaré dual to $[\omega_1]$ and
$[\omega_2]$ are represented by the following immersed oriented submanifolds
$\wh W_1(a)$ and $\wh W_2(b)$, respectively, for any $a, b \in \T^4$:
\begin{align*}
   \wh W_1(a) &:= \{x \in \T^4 \with x_1-a_1 = x_2-x_3-a_2 = 0\} \\
 & \quad \cup \{x \in \T^4 \with x_3+x_2-a_3 = x_4-a_4 = 0\}, \\
   \wh W_2(b) &:= \{x \in \T^4 \with x_1-b_1 = x_2+x_3-b_3 = 0\} \\
 & \quad \cup \{x \in \T^4 \with x_3-x_2+b_2 = x_4-b_4 = 0\}.
\end{align*}
(The ordered set of equations given for each piece determines the orientation.)
Each cycle $\wh W_1(a)$ consists of two linear tori which are both symplectic
for $\omega_1$ and Lagrangian for $\omega_2$, and which intersect positively (in
exactly two points). Thus, $\wh W_1(a)$ is an immersed symplectic submanifold in
$(\T^4,\omega)$ with positive transverse double points. By a standard procedure
(an embedded connected sum localized near each double point), $\wh W_1(a)$ can
be desingularized to an embedded and homologous symplectic submanifold $W_1(a)$
in $(\T^4,\omega)$. Similarly, $\wh W_2(b)$ can be desingularized to an embedded
symplectic submanifold $W_2(b)$ in $(\T^4,\omega)$. Moreover, since $\wh W_1(a)$
and $\wh W_2(b)$ are disjoint for $a \ne b$, so are $W_1(a)$ and $W_2(b)$. 
Therefore, if $a \ne b$, the union $W := W_1(a) \cup W_2(b)$ is a disconnected
symplectic submanifold of $(\T^4,\omega)$ whose homology class is Poincaré dual
to $2[\omega]$; in other words, $W$ is a symplectic hyperplane section of degree
$2$. To obtain a symplectic hyperplane section of degree $2k$, just replace each
linear torus involved in the definition of $\wh W_1(a)$ and $\wh W_2(b)$ by $k$
parallel copies.
\end{proof}

\subsection{Symplectic hyperplane sections and Weinstein domains} \label{s:wc}

This section is devoted to the proof of Theorem \ref{t:wc}, and we will assume
that the reader is familiar with the techniques introduced by Donaldson in \cite
{Do1,Do2} and further developed by Auroux, notably in \cite{Au2,Au3}. Actually,
the proof of Theorem \ref{t:wc} is a variation on Donaldson's proof of Theorem
\ref{t} and we will only explain the extra arguments we need (a sketch of proof
can already be found in \cite{Gi}). We recall the setting:
\begin{itemize}
\item
$V$ is a closed manifold, $\omega$ a symplectic form on $V$ with integral
periods, $J$ an $\omega$-compatible almost complex structure and $g$ the metric
given by $g(.,.) := \omega(.,J.)$; 
\item
$L \to V$ is a Hermitian line bundle whose Chern class is a lift of $[\omega]$
and $\nabla$ is a unitary connection on $L$ with curvature form $-2\pii\omega$;
\item
$\nabla',\nabla''$ are the $J$-linear and $J$-antilinear components of $\nabla$,
respectively;
\item
$L^k$, for any integer $k$, is the $k$-th tensor power of $L$ endowed with the
connection induced by $\nabla$, which we still write $\nabla = \nabla'+\nabla''$
and whose curvature form is $-2k\pii\omega$;
\item
$g_k$, for $k \ge 1$, is the rescaled metric $g_k := kg$.
\end{itemize}
In \cite{Do1}, each symplectic hyperplane section of Theorem \ref{t} is obtained
as the zero set $W := \{s_k=0\}$ of a section $s_k \map V \to L^k$, where the
sections $s_k$, $k \gg 0$, satisfy the following properties (that we formulate
using Auroux's terminology \cite{Au2}):
\begin{itemize}
\item
The sections $s_k \map V \to L^k$ are \emph{asymptotically holomorphic}. This
means that there is a positive constant $R$ such that, for every $k$, for $0
\le j \le 2$ and at every point of $V$,
$$ |s_k| \le R, \quad |\nabla^{j+1} s_k|_{g_k} \le R \quad \text{and} \quad
   |\nabla^j\nabla''s_k|_{g_k} \le Rk^{-1/2} $$
Note that the derivatives $\nabla^{j+1}s_k$ and $\nabla^j\nabla''s_k$ with $j >
0$ involve both the connection $\nabla$ on $L^k$ and the Levi-Civita connection
of the metric $g_k$ (or $g$).
\item
The sections $s_k \map V \to L^k$ are \emph{uniformly transverse} (to $0$). This
means that there is a positive constant $\eta$ such that, for every sufficiently
large integer $k$, 
$$ |\nabla s_k(x)|_{g_k} \ge \eta \quad
   \text{at every point $x$ where $|s_k(x)| \le \eta$.} $$ 
\end{itemize}
A key point here is that any section $s_k \map V \to L^k$ satisfying the above
estimates with $k > 4R^2/\eta^2$ automatically also satisfies $|\nabla''s_k| <
|\nabla's_k|$ at every point of $W = \{s_k=0\}$, and this inequality guarantees
that $W$ is a symplectic submanifold. To prove Theorem \ref{t:wc}, we will need
a similar inequality all over $V$:

\begin{definition}[Quasiholomorphic Sections] \label{d:qhs}
Let $\kappa \in [0,1)$. We will say that a section $s_k \map V \to L^k$ is \emph
{$\kappa$-quasiholomorphic} if $|\nabla''s_k| \le \kappa\,|\nabla's_k|$ at
every point of $V$.  
\end{definition}

The geometric significance of this notion is the following:

\begin{lemma}[Quasiholomorphic Sections and Symplectic Convexity] \label{l:qhs}
Let $W$ be the zero set of a $\kappa$-quasiholomorphic section $s \map V \to
L^k$, $\kappa \in [0,1)$. Then the function
$$ \phi := -\log|s| \map V - W \to \R $$
admits a Liouville pseudogradient, namely the vector field $\dvf\lambda$ where
$-2k\pii\lambda$ is the potential $1$-form of $\nabla$ in the trivialization
$s/|s|$ on $V-W$.  
\end{lemma} 

As a consequence, if $s$ vanishes transversely and if $\phi := -\log|s|$ is a
Morse function, then $W$ is a symplectic hyperplane section and the Liouville
compactification of $V - W$ (see Proposition \ref{p:lc}) is a Weinstein domain.

\begin{proof}
Setting $\rho := |s|$, we have
\begin{align*}
   2\nabla' s &= d\rho - J^*\lambda\rho - i\,J^* (d\rho - J^*\lambda\rho), \\ 
   2\nabla''s &= d\rho + J^*\lambda\rho + i\,J^* (d\rho + J^*\lambda\rho), 
\end{align*}
so
\begin{align*}
   |\nabla' s| &= \tfrac12 \, |d\rho - J^*\lambda \rho|, \\ 
   |\nabla''s| &= \tfrac12 \, |d\rho + J^*\lambda \rho| .
\end{align*}
Since $s$ is $\kappa$-quasiholomorphic, we have $|\nabla''s| \le \kappa\, 
|\nabla's|$ and we obtain (after dividing by $\rho$):
$$ |\lambda + J^*d\phi| \le \kappa\, |\lambda - J^*d\phi| . $$
Now the derivative of $\phi$ along the Liouville field $\dvf\lambda$ is equal to
the inner product $g_k(\lambda,\dJ\phi)$. Thus, for $\kappa \in [0,1)$, the
above inequality implies that
$$ \dvf\lambda \cdot \phi \ge \tfrac12 \frac{ 1-\kappa^2 }{ 1+\kappa^2 }
   \left(|\lambda|_{g_k}^2 + |d\phi|_{g_k}^2\right). $$
This shows that $\dvf\lambda$ is a pseudogradient of $\phi$.
\end{proof}

With this lemma in mind, it suffices to show:

\begin{proposition}[Construction of Quasiholomorphic Sections] \label{p:qhs}
Let $s_k^0$ be asymptotically holomorphic and uniformly transverse sections $V
\to L^k$, and let $\kappa$ be any number in $(0,1)$. Then there exist
$\kappa$-quasiholomorphic sections $s_k \map V \to L^k$ such that, for every
sufficiently large integer $k$:
\begin{itemize}
\item
the section $s_k$ vanishes transversely and the symplectic hyperplane section
$W := \{s_k=0\}$ is Hamiltonian isotopic to $W^0 := \{s_k^0=0\}$; 
\item
the function $-\log|ss_k| \map V-W \to \R$ is a Morse function.
\end{itemize}
\end{proposition}

The main step in the proof is the next Lemma which provides asymptotically
holomorphic sections of $L^k$ satisfying more uniform transversality conditions.   
We recall that, given a positive number $\eta$, a Riemannian manifold $M$ and a
Hermitian vector bundle $E \to M$ endowed with a unitary connection $\nabla$, a
section $\sigma \map M \to E$ is \emph{$\eta$-transverse} (to $0$) if, at every
point $x \in M$ with $|\sigma(x)| \le \eta$, the linear map $\nabla\sigma(x)
\map T_xM \to E_x$ is surjective and has a right inverse whose operator norm
does not exceed $1/\eta$. If the real rank of $E$ equals the dimension of $M$,
it is equivalent to require that $|\nabla\sigma(x) \cdot v| \ge \eta\, |v|$ for
all vectors $v \in T_xM$.

In what follows, we consider sections $\sigma_k \map V \to E \otimes L^k$, where
$E \to V$ is a fixed Hermitian bundle and $k$ runs over all sufficiently large
integers, and we say that these sections are \emph{uniformly transverse} if they
are $\eta$-transverse for some positive $\eta$ independent of $k$, where the 
amount of transversality is measured with the metric $g_k$.

\begin{lemma}[Extra Uniform Transversality Condition] \label{l:xutc}
Let $s_k^0$ be asymptotically holomorphic and uniformly transverse sections $V
\to L^k$. For large integers $k$, the sections $s_k^0$ are homotopic, through
asymptotically holomorphic and uniformly transverse sections, to sections
$s_k^1 \map V \to L^k$ whose partial covariant derivatives $\nabla's_k$ are
uniformly transverse.
\end{lemma}

\begin{proof}[Proof of the lemma]
The proof follows step by step the path opened by Donaldson in \cite{Do1}. We
just explain here how to obtain uniform local transversality for sections of the
form $\nabla's_k$. The globalization process elaborated by Donaldson in \cite
{Do1} then applies readily to provide the desired sections $s_k^1$. The sections
$s_k^1$ will be asymptotically holomorphic by construction. Moreover, Given any
$\delta>0$, we can arrange that all the differences $s_k^1-s_k^0$ are bounded by
$\delta$ in $\cC^1$-norm. For $\delta$ smaller than the uniform transversality
modulus of the sections $s_k^0$, it follows that, for every $t \in [0,1]$, the
sections $(1-t)s_k^0+ts_k^1$ are still asymptotically holomorphic and uniformly 
transverse.

To achieve uniform local transversality, we essentially need to show that the
derivatives $\nabla's_k^0$ are represented (in Darboux coordinates independent
of $k$ and in balls of fixed $g_k$-radius) by maps which (on smaller balls) are
approximated within $\eps$ in $\cC^1$-norm by polynomial maps of degree bounded
by $C \log(1/\eps)$, where $C$ is a positive constant (independent of $k$).

We work in complex Darboux coordinates $(z_1 \etc z_n)$ centered on a point $a$,
with the trivialization of $L^k$ given by parallel translation along rays. We
denote by $J_0$ the standard complex structure in these coordinates and by
$\nabla'_0, \nabla''_0$ (resp. $d', d''$) the $J_0$-linear and $J_0$-antilinear 
components of $\nabla$ (resp.~of the usual differential $d$). Thus we have
$$ \nabla's_k^0 - \nabla'_0s_k^0 = -\tfrac i2 \nabla s_k^0 \circ (J-J_0) $$
where the right-hand side, measured with the metric $g_k$ on a ball of fixed
radius, is bounded by $O(k^{-1/2})$ in $\cC^1$-norm. Hence it suffices to make
the partial covariant derivatives $\nabla'_0s_k$ uniformly transverse to $0$,
and for this we can use the connection of the flat metric rather than that of
$g_k$. Note that there is a little subtlety here: we want $\nabla's_k$ to be
transverse to $0$ as a section of $T'V \otimes L^k$ ($T'V$ denoting the space of
$J$-linear covectors in $T^*V \otimes \C$), but $\nabla's_k$ and $\nabla'_0s_k$ 
are not sections of the same bundle. To derive the transversality of $\nabla'
s_k$ from that of $\nabla'_0s_k$, we observe that transversality between spaces
of equal dimensions is a dilation property for all non-zero vectors (under the
differential) and this property is stable under $\cC^1$-small perturbations.

Let $s_{a,k}$ be the Gaussian section of $L^k$ at $a$. Since we work in a ball
of given radius, for $k$ sufficiently large,
$$ s_{a,k}(z) = \exp (-\pi|z|^2/2) . $$  
There are two obvious bases in the space of $J_0$-linear forms, one consisting
of the forms $dz_j s_{a,k}$ and one consisting of the forms $\nabla'_0 (z_j
s_{a,k})$. They are related by
\begin{align*}
   \nabla'_0 (z_i s_{a,k})
&= dz_i s_{a,k} + z_i \nabla'_0s_{a,k} \\
&= \left(dz_i - \pi z_i \sum_j \ol z_j dz_j\right) s_{a,k} \\
&= \sum_j \Phi_{ij}(z) \, dz_j s_{a,k}
\end{align*}
where the entries of the matrix
$$ \Phi(z) = \bigl(\Phi_{ij}(z)\bigr)
 = \bigl(\delta_{ij} - \pi z_i \ol z_j\bigr) $$
are (real) polynomials independent of $k$.  

We now represent $\nabla'_0 s_k^0$ by the map $h = (h_1 \etc h_n)$ (with values
in $\C^n$) defined by
$$ \nabla'_0 s_k^0 = \sum_j h_j \nabla'_0 (z_j s_{a,k}). $$
If $w = (w_1 \etc w_n)$ is a $\delta$-transverse value of $h$ (meaning that
$h-w$ is $\eta$-transverse to $0$) then the section
$$ \nabla'_0 \left(s_k^0 - \sum_j w_j z_j s_{a,k}\right)
 = \sum_j (h_j - w_j) \nabla'_0 (z_j s_{a,k}) $$ 
is $\eta'$-transverse to $0$ for some $\eta'$ which is a definite fraction of
$\eta$. On the other hand, considering the function $f = s_k^0 / s_{a,k}$, we
have 
\begin{align*}
   \nabla'_0 s_k^0
&= d'f\, s_{a,k} + f\, \nabla'_0 s_{a,k} \\
&= \left(d'f - \pi f \sum_i \ol z_i dz_i\right) s_{a,k} \\
&= \sum_i (\del_{z_i}f - \pi f \ol z_i) dz_i s_{a,k} .
\end{align*}
In other words, if we denote by $u = (u_1 \etc u_n)$ the map given by
$$ u_i := \del_{z_i}f - \pi \ol z_if, \quad
   1 \le i \le n, $$
we get
$$ h(z) = \Phi(z)^{-1} u(z) . $$ 
Since the function $f$ is approximately holomorphic and the entries of the
matrix $\Phi^{-1}$ are analytic functions independent of $k$, the map $h$ admits
the required polynomial approximations (see \cite{Do1} for more details).
\end{proof}

\begin{remark*}[Cheaper Approach]
The above argument appeals (implicitly) to the quantitative version of Sard's
theorem given in \cite[Section 5]{Do2} or, more accurately, to its real version
proved in \cite[Section 6]{Mo1}. This is a great result but its proof is
difficult and quite technical. One could modify our argument to appeal, instead,
to the trick proposed by Auroux in \cite{Au3}. This would definitely make the
complete proof of Theorem \ref{t:wc} technically much simpler, but it would make
our exposition here more intricate. 
\end{remark*}

\begin{proof}[Proof of Proposition \ref{p:qhs}]
First observe that, since the sections $s_k^0$ and $s_k^1$ are homotopic through
asymptotically holomorphic and uniformly transverse sections, their zero sets
$W^0 := \{s_k^0=0\}$ and $W^1 := \{s_k^1=0\}$ are Hamiltonian isotopic. We will
now construct $\kappa$-quasiholomorphic sections $s_k$ by modifying the sections
$s_k^1$ away from their zero sets. Hence, the symplectic hyperplane sections
$W := \{s_k=0\} = W^1$ and $W^0$ will remain Hamiltonian isotopic for every
large integer $k$.

Consider the sets $\Gamma_k \supset \Delta_k$ defined by
\begin{align*}
   \Gamma_k
&= \{x \in V \with |\nabla''s_k^1(x)| \ge \kappa |\nabla's_k^1(x)|\}, \\ 
   \Delta_k &= \{x \in V \with \nabla's_k^1(x) = 0\}, 
\end{align*}
where the sections $s_k^1 \map V \to L^k$ are those given by the lemma.

Since $s_k^1$ vanishes $\eta$-transversely, $\Gamma_k$ avoids a tube of fixed
$g_k$-radius (independent of $k$) about $W^1 := \{s_k^1=0\}$. Moreover, since
$\nabla's_k^1$ vanishes $\delta$-transversally, $\Delta_k$ is a discrete (hence
finite) set and:

\begin{lemma}[Location of Bad Points] \label{l:Gam_k}
For every sufficiently small positive number $\rho$ and every sufficiently large
integer $k \ge k(\rho)$, the balls $B_k(a,\rho)$, $a \in \Delta_k$, are disjoint
and cover $\Gamma_k$.
\end{lemma}
 
As in \cite[Lemma 8 and Proposition 9]{Do2}, this lemma is a consequence of the
following simple fact:

\begin{lemma}[Inverse Function Theorem] \label{l:ift}
Let $\phi \map \D^n \to \R^n$ be a map $\cC^2$-bounded by $c$ and such that
$$ |d\phi(0) \cdot v| \ge \delta\, |v| \quad \text{for all vectors $v$.} $$
If $|\phi(0)| \le \delta\rho/2$ for some $\rho \le \delta/c$, the equation 
$\phi(x) = 0$ has a unique solution $x$ in the ball of radius $\rho$ about $0$.
\end{lemma}

To prove Lemma \ref{l:Gam_k}, we apply Lemma \ref{l:ift} to the map representing
$\nabla's_k^1$ in the complex Darboux coordinates centered on a point $a$ of
$\Gamma_k$. At this point,
$$ |\nabla's_k^1(a)| \le \kappa^{-1} |\nabla''s_k^1(a)| \le
   R \kappa^{-1} k^{-1/2} $$
so the hypotheses of Lemma \ref{l:ift} are fulfilled once $k$ is sufficiently
large.

\medskip

To complete the proof of the proposition, we will modify $s_k^1$ near each point
$a \in \Delta_k$ (see \cite[Lemma 10 and the subsequent discussion]{Do2}).
Again, we work in the complex Darboux coordinates centered on $a$. For any
$\rho>0$, fix a cutoff function $\beta = \beta_\rho$ such that $\beta(z) = 1$
for $|z| \le \rho/2$, \ $\beta(z) = 0$ for $|z| \ge \rho$, and $|d\beta(z)| \le
3/\rho$ for all $z$. Write $s_k^1 = f s_{a,k}$ and denote by $f_0$ the complex
polynomial of degree $2$ given by
$$ f_0(z) = f(0) + \tfrac12 \sum_{ij}
   \del^2_{z_iz_j} f(0) z_iz_j . $$    
We then consider the sections $s_k$ defined in the coordinates $(z_1 \etc z_n)$
by
$$ s_k := \bigl(\beta f_0 + (1-\beta) f\bigr) s_{a,k}. $$
Before comparing the derivatives $\nabla's_k$ and $\nabla''s_k$, let us compare
the derivatives $\nabla'_0s_k$ and $\nabla''_0s_k$. As we already noticed, the
closeness of $\nabla's_k^1$ and $\nabla'_0s_k^1$ guarantees that the latter
derivative is $\eta/2$-transverse to $0$ on the ball of radius $\rho$ for $k$
sufficiently large. On the other hand, the identities
\begin{align*}
   d''f(0) &= \nabla''_0 s_k^1 (0), \\
   dd''f(0) &= \nabla_0\nabla''_0 s_k^1 (0) 
\end{align*}
(where $\nabla_0$ denotes the connection associated to the flat metric) show
that $|d''f(0)|$ and $|dd''f(0)|$ are bounded by $C k^{-1/2}$. Therefore, if $k$
is sufficiently large, the partial derivative $\nabla'_0 (f_0 s_{a,k})$ is so 
close to $\nabla'_0 s_k^1$ that it is $\eta/4$-transverse to $0$. Furthermore,
$f_0 s_{a,k}$ is a holomorphic section. Thus, on the ball of radius $\rho/2$
(where $\beta=1$), we have
$$ \nabla''_0 s_k(z) = 0 \quad \text{and} \quad
   |\nabla'_0 s_k(z)| \ge \frac\eta4 |z| . $$
Hence, on that same ball,
$$ |\nabla''s_k(z)| \le Ck^{-1/2} |z| \quad \text{and} \quad
   |\nabla'_0 s_k(z)| \ge \Bigl(\frac\eta4 - Ck^{-1/2}\bigr) |z| . $$
In the annular region $\rho/2 \le |z| \le \rho$, the calculations above imply
that
$$ |f(z) - f_0(z)| \le C (\rho^3 + \rho k^{-1/2}) $$ 
and, since the gradient of $\beta$ is bounded by $3/\rho$, the same arguments as
in \cite{Do2} give the desired inequalities when $\rho$ is sufficiently small.

It remains to show that the function $\phi := -\log|s_k| \map V-W \to \R$ (where
$W := \{s_k=0\}$) is a Morse function. Since $s_k$ is $\kappa$-quasiholomorphic
with $\kappa < 1$, the critical points of $\phi$ are the zeros of $\nabla's_k$,
namely the points of $\Delta_k$. It then follows form the porperties of $s_k$ in
$B_k(a,\rho/2)$, $a \in \Delta_k$, that $\nabla s_k$ vanishes transversele at
$a$, so the critical points of $\phi$ are non-degenerate.
\end{proof}

\subsection{Symplectic hyperplane sections and Stein domains}

Here we derive Theorem \ref{t:sc} from Theorem \ref{t:wc}. The
main ingredient we will use is a special case (a domain is a cobordism with
empty bottom boundary) of \cite[Theorem 13.5]{CE}:

\begin{theorem}[Cieliebak-Eliashberg] \label{t:CE-13.5}
Let $(F,\lambda)$ be a Weinstein domain and $\phi_0$ a function on $F$ with
pseudogradient $\dvf\lambda$ and regular level set $\del F = \{\phi_0=0\}$.
Then there exist a complex structure $J$ and a path of $1$-forms $\lambda_t$ on
$F$ $(t \in [0,1])$ with the following properties:
\begin{itemize}
\item
all forms $d\lambda_t$ are symplectic on $F$, and $\lambda_0 = \lambda$;
\item
all Liouville vector fields $\dvf{\lambda_t}$ are pseudogradients of $\phi_0$;
\item
$\lambda_1 = \dJ (u \circ \phi_0)$ for some convex increasing function $u \map
\R_{\le0} \to \R_{\le0}$ with $u(0) = 0$.  
\end{itemize}
\end{theorem}

In particular, $(F,J)$ is a Stein domain and $u \circ \phi$ is a $J$-convex
function.
 
To complete the proof of Theorem \ref{t:sc}, we actually need a variant of the
above result, namely:

\begin{corollary}[Weinstein and Stein Domains] \label{c:wsd}
Let $(F,\lambda)$ be a Weinstein domain. Then there exist a complex structure
$J$ on $F$ and a $J$-convex Morse function $\phi \map F \to \R_{\le0}$, with
regular level set $\del F = \{\phi=0\}$, such that $d\lambda = \ddJ\phi$.
\end{corollary}

\begin{proof}
Pick an arbitrary function $\phi_0$ on $F$ with pseudogradient $\dvf\lambda$
and regular level set $\del F = \{\phi_0=0\}$. Consider the complex structure
$J$ and the path of $1$-forms $\lambda_t$ (along with the function $u$) given by
Theorem \ref{t:CE-13.5}. Since the Liouville vector fields $\dvf{\lambda_t}$ are
all pseudogradients of $\phi_0$, each form $\lambda_t$ induces a contact form
$\alpha_t$ on $\del F$. Using Gray's stability theorem and a suitable isotopy
extension, we can arrange that the forms $\lambda_t$ have the same kernel along
$\del F$, \emph{i.e.} $\lambda_t = v_t \lambda_0$ on $\del F$ for some function
$v_t \map \del F \to \R_{>0}$.

Assume temporarily that $v_t=1$ for all $t$. Then Moser's argument provides an
isotopy $h_t$ of $F$ relative to $\del F$ such that $h_0 = \id$ and $h_t^*
d\lambda_t = d\lambda$. Then the complex structure $h_1^*J$ and the function
$h_1^* (u \circ \phi_0)$ have the desired properties.

Therefore it suffices to modify the forms $\lambda_t$ so that they coincide on
(or along) $\del F$ and still satisfy the conditions of Theorem \ref{t:CE-13.5}.
It is easy to find positive functions $w_t$ on $F$ such that $w_t = 1/v_t$ on
$\del F$ and $\dvf{\lambda_t} \cdot \log w_t > -1$. Then the forms $\wt\lambda_t
:= w_t \lambda_t$ agree along $\del F$ and satisfy the first two conditions of
Theorem \ref{t:CE-13.5}, but $\wt\lambda_1$ and $\dJ(u \circ \phi_0)$ are not
equal. Set $\phi_1 = u \circ \phi_0$ and note that
$$ \wt\lambda_1 = w_1 \lambda_1 = w_1 \dJ\phi_1. $$  
Lemma \ref{l:Jcf} below provides a function $\phi$ such that $\wt\lambda_1 =
\dJ\phi$, which completes the proof of the corollary. 
\end{proof}

\begin{lemma}[Rescaling of $J$-Convex Functions] \label{l:Jcf}
Let $F$ be a domain, $J$ a complex structure on $F$ and $\phi_1 \map F \to \R$ a
$J$-convex Morse function on $F$ with regular level set $\del F = \{\phi_1=0\}$.
For every positive function $w$ on $\del F$, there exists a $J$-convex Morse
function $\phi \map F \to \R$ equivalent to $\phi_1$ such that $\dJ\phi = w\,
\dJ\phi_1$ along $\del F$.
\end{lemma}

By ``equivalent'', we mean that $\phi = u \circ \phi_1 \circ f$, where $u \map
\R \to \R$ is an increasing function while $h$ is a diffeomorphism of $F$.  

\begin{proof}
First extend $w$ to a positive function on $F$ and define $\phi_2 := (w+c\phi_1)
\phi_1$, where $c$ is a positive constant. Then $\del F$ is a regular component
of the zero-level set of $\phi_2$, and $\dJ\phi_2 = w\,\dJ\phi_1$ at every point
of $\del F$. Moreover,
\begin{multline*}
   \ddJ\phi_2 = (w+c\phi_1)\, \ddJ\phi_1 + \phi_1\, \ddJ (w+c\phi_1) \\
 + dw \wedge \dJ\phi_1 + d\phi_1 \wedge \dJ w + 2c\,d\phi_1 \wedge \dJ\phi_1,
\end{multline*}
so $\phi_2$ is $J$-convex near $\del F$ for any sufficiently large constant $c$.
We henceforth fix such a $c$. Then there exists a number $\delta>0$ such that
$d\phi_2$ is positive on the Liouville field $\dvf{\dJ\phi_1}$ in the collar
$\{-\delta \le \phi_1 \le 0\}$ (indeed, $d\phi_2 = w\,d\phi_1$ at every point
of $\del F$). Now set 
$$ \phi_3 = a\phi_1+b \quad \text{with} \quad 
   b := \tfrac12 \sup \{\phi_2(x) \with \phi(x) = -\delta\}, \quad
   a < \frac b \delta. $$
Clearly, $\phi_3$ is $J$-convex and we obtain the desired function $\phi$ by
smoothing the function $\max (\phi_1,\phi_2)$ (see \cite[Chapter 2]{CE} for
details on the relevant smoothing technique).
\end{proof}

\end{document}